\documentclass{amsart}

\title
{Geometrical formality of solvmanifolds and solvable Lie type geometries}

\author{Hisashi Kasuya}
\usepackage{amssymb}
\usepackage{amsmath}
\usepackage{amscd}
\usepackage{amstext}
\usepackage{amsfonts}
\usepackage[all]{xy}

\newcommand{\C}{\mathbb{C}}
\newcommand{\R}{\mathbb{R}}

\newcommand{\Q}{\mathbb{Q}}
\newcommand{\Z}{\mathbb{Z}}
\newcommand{\g}{\frak{g}}

\theoremstyle{plain}
\newtheorem{theorem}{Theorem}[section]
\theoremstyle{plain}
\newtheorem{remark}{Remark}[section]
\theoremstyle{lemma}

\theoremstyle{definition}
\newtheorem{definition}[theorem]{Definition}
\theoremstyle{proposition}

\theoremstyle{corollary}
\newtheorem{corollary}[theorem]{Corollary}
\theoremstyle{remark}
\newtheorem{example}{\bf Example}
\begin{document} 

\maketitle
\begin{abstract}
We show that for a Lie group $G=\R^{n}\ltimes_{\phi} \R^{m}$ with a semisimple action $\phi$ which has a cocompact discrete subgroup $\Gamma$, the solvmanifold $G/\Gamma$   admits a canonical
invariant formal (i.e.  all products of harmonic forms are again harmonic) metric.
We show that a compact oriented aspherical manifold of dimension less than or equal to $4$ with the virtually solvable fundamental group admits a formal metric if and only if it is diffeomorphic to a torus or an infra-solvmanifold which is not a nilmanifold.
\end{abstract}
\section{Introduction}
Let $(M, g)$ be a compact oriented Riemannian $n$-manifold.
We call $g$ formal if all products of harmonic forms are again harmonic.  
If a compact oriented manifold admits a formal Riemaniann metric, we call it geometrically formal.
If $g$ is formal, then the space of the harmonic forms is a subalgebra of the de Rham complex of $M$ and isomorphic to the real cohomology of $M$.
By this, a geometrically formal manifold is a formal space (in the sense of Sullivan \cite{Sul}).
But the converse is not true (see \cite{Kot} \cite{KT}).
For very simple examples, closed surfaces with genus$\ge2$ are formal but not geometrically formal.
Thus we have one problem of geometrical formality of formal spaces.
Kotschick's nice  work in \cite{Kot} stimulates us  to consider this problem.

In this paper we prove the following theorem by using  computations of the de Rham cohomology of general solvmanifolds given in \cite{K2}.
\begin{theorem}\label{MT}
Let $G=\R^{n}\ltimes_{\phi} \R^{m}$ with a semisimple action $\phi$.
Suppose $G$ has a lattice $\Gamma$.
Then $G/\Gamma$ admits an invariant formal metric.
\end{theorem}

We also study geometrical formality of low-dimensional aspherical  manifolds with the virtually  solvable fundamental groups.
We consider infra-solvmanifolds which are quotient spaces of simply connected solvable Lie groups by subgroups of the groups of the affine transformations of $G$ satisfying some conditions (see Section $7$ for the definition). 
We classify geometrically formal compact aspherical manifolds of dimension less than or equal to $4$ with the virtually solvable fundamental groups.
\begin{theorem}
Let $M$ be a compact oriented aspherical manifold of dimension less than or equal to $4$ with the virtually solvable fundamental group.
Then  $M$ is geometrically formal if and only if $M$ is diffeomorphic to a torus or an infra-solvmanifold which is not a nilmanifold.
\end{theorem}

\section{Notation and conventions}

Let $k$ be a subfield of $\C$.
A group $\bf G$ is called a $k$-algebraic group if $\bf G$ is a Zariski-closed subgroup of $GL_{n}(\C)$ which is defined by polynomials with coefficients in $k$.
Let  ${\bf G}(k)$ denote the set  of  $k$-points of $\bf G$ 
and ${\bf U}({\bf G})$ the maximal Zariski-closed unipotent normal $k$-subgroup of $\bf G$ called the unipotent radical of $\bf G$.
Denote $U_{n}(k)$ the group of $k$-valued upper triangular unipotent matrices of size $n$.

\section{Unipotent hull of solvable Lie group}\label{Hul}
\begin{theorem}{\rm (\cite{R})}\label{ttt}
Let $G$ be a simply connected solvable Lie group.
Then there exists a unique $\R$-algebraic group ${\bf H}_{G}$ with an injective group homomorphism $\psi :G\to {\bf H}_{G}(\R)  $ 
so that:
\\
$(1)$  \ $\psi (G)$ is Zariski-dense in ${\bf H}_{G}$.\\
$(2)$ \    The centralizer $Z_{{\bf H}_{G}}({\bf U}({\bf H}_{G}))$ of ${\bf U}({\bf H}_{G})$ is contained in  ${\bf U}({\bf H}_{G})$.\\
$(3)$ \ $\dim {\bf U}({\bf H}_{G})$=${\rm dim}\,G$(resp. ${\rm rank}\, G$).   
\end{theorem}
We denote ${\bf U}_{G}={\bf U}({\bf H}_{G})$.
\begin{theorem}\label{abab} {\rm (\cite{K})}
 Let $G$ be a simply connected solvable Lie group.
Then ${\bf U}_{G}$ is abelian if and only if  $G=\R^{n}\ltimes_{\phi} \R^{m}$ such that the action $\phi:\R^{n}\to {\rm  Aut} (\R^{m})$ is semisimple.
\end{theorem}
\section{Hodge theory}
Let $(V, g) $ be a $\R$ or $\C$-vector space of dimension $n$ with an inner product $g$.
Let $\bigwedge V =\bigoplus_{p=0} \bigwedge^{p} V$ be the exterior algebra of $V$.
We extend $g$ to the inner product on  $\bigwedge V$.
Take $vol \in  \bigwedge^{n} V$ such that $g(vol, vol)=1$.
We define the linear map $\ast_{g}:\bigwedge^{p} V\to \bigwedge^{n-p} V$ as:
\[v \wedge \ast_{g}\bar u=g(v,u)vol\]
Let $\{\theta_{1},\dots \theta_{n}\}$ be an orthonormal basis of  $(V, g) $.
Then we have 
\[\ast_{g}(\theta_{i_{1}}\wedge\dots \theta_{i_{p}})=({\rm sgn}\sigma_{IJ}) \theta_{j_{1}}\wedge\dots \theta_{j_{n-p}}\]
 where  $J=\{j_{1}, \cdots ,j_{n-p}\}$ is the complement of $I=\{i_{1},\dots ,i_{p}\}$ in $\{1,\dots, n\}$ and $\sigma_{IJ}$ is the permutation $\begin{pmatrix}
1 \cdots p& p+1 \cdots  n \\
i_{1} \cdots i_{p}&j_{1} \cdots j_{n-p}
\end{pmatrix}
$.

Let $(M, g)$ be a compact oriented Riemannian $n$-manifold.
Let $(A^{\ast}(M), d)$ be the de Rham complex of $M$ with the exterior derivation $d$.
For $x\in M$ by the inner product $g_{x}$ on $T_{x}M$ we define the linear map $\ast_{g}:A^{p}(M)\to A^{n-p}(M)$ by 
\[(\ast_{g}(\omega))_{x}=\ast_{g_{x}}\omega_{x}    \]
for $\omega \in A^{p}(M)$.
Define  $\delta :A^{p}(M)\to A^{p-1}(M)$ by $\delta =(-1)^{np+n+1}\ast_{g} d\ast_{g}$.
We call $\omega\in A^{p}(M)$ harmonic if $d\omega =0$ and $\delta \omega =0$.
Let ${\mathcal H}^{p}(M)$ be the subspace of $A^{p}(M)$ which consists of harmonic $p$-forms.
Let ${\mathcal H}(M)=\bigoplus{\mathcal H}^{p}(M)$.
It is known that the inclusion ${\mathcal H}(M)\subset A^{\ast}(M)$ induces an isomorphism
\[{\mathcal H}^{p}(M)\cong H^{p}(M,\R).\]
In general a wedge product of harmonic forms is not harmonic and so ${\mathcal H}^{p}(M)$ is not a subalgebra of $A^{\ast}(M)$.
\begin{definition}
We call  a Riemannian metric $g$ formal if all products of harmonic forms are again harmonic.
We call an oriented compact manifold $M$ geometrical formal if $M$ admits a formal metric.
\end{definition}

\section{Invariant forms on solvmanifolds (proof of Theorem \ref{MT})}
Let $G$ be a simply connected solvable Lie group and $\g$  the Lie algebra which is the space of the left invariant vector fields on $G$.
Consider the exterior algebra $\bigwedge \g^{\ast} $ of the dual space of $\g$.
Denote $d:\bigwedge^{1}\g^{\ast}\to \bigwedge^{2}\g^{\ast}$ the dual map of the Lie bracket of $\g$ and  $d:\bigwedge^{p}\g^{\ast}\to \bigwedge^{p+1}\g^{\ast}$ the extension of this map.
We can identify  $(\bigwedge \g^{\ast}, d)$ with the left invariant forms on $G$ with the exterior derivation.
Let ${\rm Ad}: G\to {\rm Aut}(\g)$ be the adjoint representation.
Representations of $G$ are trigonalizable in $\C$ by Lie's theorem.
We define the diagonal representation ${\rm Ad}_{s}:G\to {\rm Aut}(\g_{\C}) $ as the diagonal entries of a triangulation of $\rm Ad$.
Let $X_{1},\cdots ,X_{n}$ be a basis of $\g_{\C}$ such that ${\rm Ad}_{s}$ is represented by diagonal matrices.
Then we have ${\rm Ad}_{sg}X_{i}=\alpha_{i}(g)X_{i}$ for characters $\alpha_{i}$ of $G$.
Let $x_{1},\dots,x_{n}$ be tha dual basis of $X_{1},\dots ,X_{n}$.
We assume that $G$ has a lattice $\Gamma$.
Define the sub-DGA $A^{\ast}$ of the de Rham complex $A^{\ast}_{\C}(G/\Gamma)$ as

\begin{multline*}
A^{p}
=\left\langle \alpha_{i_{1}\dots i_{p}} x_{i_{1}}\wedge \dots \wedge x_{i_{p}} {\Big \vert} \begin{array}{cc}1\le i_{1}<i_{2}<\dots <i_{p}\le n,\\  {\rm the \, \,  restriction} \, \,  of \, \, \alpha_{i_{1}\dots i_{p}}\, \, {\rm on  \, \, \Gamma \, \, is\, \,  trivial}\end{array}\right\rangle
\end{multline*}
where $\alpha_{i_{1}\dots i_{p}}=\alpha_{i_{1}}\dots \alpha_{i_{p}}$. 

\begin{theorem}\label{isoc}{\rm (\cite[v4. Corollary 7.6]{K2})}
The  inclusion
\[ A^{\ast}\subset A^{\ast}_{\C}(G/\Gamma)\]
induces a cohomology isomorphism and
$A^{\ast}$ can be considered as a sub-DGA of $\bigwedge {\frak u}^{\ast}$ where ${\frak u}$ is the Lie algebra of ${\bf U}_{G}$ as in Section \ref{Hul}.

\end{theorem}

Define $g$  the Hermittian inner product as 
\[ g(X_{i}, X_{j})=\delta_{ij} .\]
Since $ {\rm Ad}_{s}$ is an $\R$-valued representation, the restriction of $g$ on $\g$ is an inner product on $\g$.
We consider $g$ as an invariant Riemannian metric on $G/\Gamma$.
\begin{theorem}
If ${\bf U}_{G}$ is abelian, then $g$ is a formal metric on $G/\Gamma$.
\end{theorem}
\begin{proof}
By the assumption, the differential on $\bigwedge {\frak u}^{\ast}$ is $0$.
By Theorem \ref{isoc}, the derivation on $A^{\ast}$ is $0$ and we have an
isomorphism
\[ A^{\ast}\cong H^{\ast}(G/\Gamma).\]
Thus it is sufficient to show that all elements of $A^{\ast}$ are harmonic.
Let $\ast_{g}$ be the star operator.
Then for $\alpha_{i_{1}\dots i_{p}} x_{i_{1}}\wedge \dots \wedge x_{i_{p}} \in A^{p}$  we have 
\[\ast_{g}(\alpha_{i_{1}\dots i_{p}} x_{i_{1}}\wedge \dots \wedge x_{i_{p}} )=({\rm sgn} \sigma_{IJ}) \bar \alpha_{i_{1}\dots i_{p}} x_{j_{1}}\wedge\dots \wedge x_{j_{n-p}}.\]
Since  the  restriction  of $ \alpha_{i_{1}\dots i_{p}}$ on  $ \Gamma $ is  trivial, the image $\alpha_{i_{1}\dots i_{p}}(G)=\alpha_{i_{1}\dots i_{p}}(G/\Gamma)$ is compact and hence $\alpha_{i_{1}\dots i_{p}}$ is unitary.
Since $G$ has a lattice $\Gamma$, $G$ is unimodular (see \cite[Remark 1.9]{R}) and hence we have 
\[\bar \alpha_{i_{1}\dots i_{p}} =\alpha_{i_{1}\dots i_{p}}^{-1}=\alpha_{j_{1}\dots j_{n-p}}.
\]
Hence we have 
\[\bar \alpha_{i_{1}\dots i_{p}} x_{j_{1}}\wedge\dots \wedge x_{j_{n-p}}=\alpha_{j_{1}\dots j_{n-p}} x_{j_{1}}\wedge\dots \wedge x_{j_{n-p}}\in  A^{n-p}\]
and thus we have $\ast_{g}(A^{\ast})\subset A^{\ast}$.
Since the derivation on $A^{\ast}$ is $0$, we have $\delta(A^{\ast})=0$.
Hence the theorem follows.

\end{proof}
By this theorem and Theorem \ref{abab}, we have Theorem \ref{MT}.
\begin{remark}
Not every invariant metric on $G/\Gamma$ in Theorem \ref{MT} is formal.
See the following example.
\end{remark}
\begin{example}
Let $H=\R\ltimes_{\phi}\R^{2}$ such that $\phi(z)(x,y)=(e^{z}x, e^{-z}y)$.
Consider $G=H\times \R$.
Then for some non-zero number $a\in \R$, $\phi(a)$ is conjugate to an element of $SL_{2}(\Z)$, and hence
 $G$ has a lattice $\Gamma=a\Z\ltimes_{\phi}\Gamma^{\prime}\times \Z$ for a lattice $\Gamma^{\prime}$ of $\R^{2}$.
Let $\g$ be the Lie algebra of $G$ and $\g^{\ast}$ the dual of $\g$.
The cochain complex $(\bigwedge \g^{\ast}, d)$ is generated by a basis $\{x,y,z,w\}$ such that 
\[dx=-z\wedge x,\  dy=z\wedge y,\  dz=0,\  dw=-z\wedge x.\]
Consider the invariant metric $g=x^2+y^2+z^2+w^2$.
Then $z$ and $w-x$ are harmonic for $g$.
But $z\wedge (w-x)$ is not harmonic.
Thus $g$ is not formal.
\end{example}
\begin{example}\label{CCC}
Let $G=\C\ltimes_{\phi} \C^{2}$ with $\phi(z)(x,y)=(e^{z}x, e^{-z}y)$.
For some $p,q\in \R$, $\phi(p\Z+\sqrt{-1}q\Z)$ is conjugate to a subgroup of $SL_{4}(\Z)$ and hence we have a lattice $\Gamma=(p\Z+\sqrt{-1}q\Z)\ltimes \Gamma^{\prime\prime}$ for  a lattice $\Gamma^{\prime\prime}$ of $\C^{2}$ (see \cite{Na} and \cite{Hd}).
For any lattice $\Gamma$, $G/\Gamma$ is geometrically formal by Theorem \ref{MT}.
\begin{remark}
In \cite{BT} for some lattice of $G$ in  Example \ref{CCC}, it is proved that $G/\Gamma$ is geometrically formal.
But the de Rham cohomology of $G/\Gamma$ varies  according to a choice of a lattice $\Gamma$.
\end{remark}
\end{example}
\begin{example}
Let $K$ be a finite extension field of $\Q$ with the degree $r$ for positive integers.
We assume $K$ admits embeddings $\sigma_{1},\dots \sigma_{s},\sigma_{s+1},\dots, \sigma_{s+2t}$ into $\C$ such that $s+2t=r$, $\sigma_{1},\dots ,\sigma_{s}$ are real embeddings and $\sigma_{s+1},\dots, \sigma_{s+2t}$ are complex ones satisfying $\sigma_{s+i}=\bar \sigma_{s+i+t}$ for $1\le i\le t$. 
We suppose $s>0$.
Denote ${\mathcal O}_{K}$ the ring of algebraic integers of $K$, ${\mathcal O}_{K}^{\ast}$ the group of units in ${\mathcal O}_{K}$ and 
\[{\mathcal O}_{K}^{\ast\, +}=\{a\in {\mathcal O}_{K}^{\ast}: \sigma_{i}(a)>0 \,\, {\rm for \,\,  all}\,\, 1\le i\le s\}.
\]  
Define $\sigma :{\mathcal O}_{K}\to \R^{s}\times \C^{t}$ by
\[\sigma(a)=(\sigma_{1}(a),\dots ,\sigma_{s}(a),\sigma_{s+1}(a),\dots ,\sigma_{s+t}(a))
\]
for $a\in {\mathcal O}_{K}$.
We denote 
\[\sigma(a)\cdot \sigma(b)=(\sigma_{1}(a)\sigma_{1}(b),\dots ,\sigma_{s}(a)\sigma_{s}(b),\sigma_{s+1}(a)\sigma_{s+1}(b),\dots ,\sigma_{s+t}(a)\sigma_{s+t}(b))\]
for $a,b\in  {\mathcal O}_{K}$.
Then the image $\sigma({\mathcal O}_{K})$ is a lattice in $\R^{s}\times \C^{t}$.
Define $l:{\mathcal O}_{K}^{\ast\, +}\to \R^{s+t}$ by 
\[l(a)=(\log \vert \sigma_{1}(a)\vert,\dots ,\log \vert \sigma_{s}(a)\vert , 2\log \vert \sigma_{s+1}(a)\vert,\dots ,2\log \vert \sigma_{s+t}(a)\vert)
\]
for $a\in {\mathcal O}_{K}^{\ast\, +}$.
Then by Dirichlet's units theorem, the image $l({\mathcal O}_{K}^{\ast\, +})$ is a lattice in the vector space $L=\{x\in \R^{s+t}\vert \sum_{i=1}^{s+t} x_{i}=0\}$.
By this we have a geometrical representation of the semi-direct product ${\mathcal O}_{K}^{\ast\, +}\ltimes  {\mathcal O}_{K}$ as $l({\mathcal O}_{K}^{\ast\, +})\ltimes_{\phi} \sigma({\mathcal O}_{K})$ with 
\[\phi(t_{1},\dots, t_{s+t})(\sigma(a))=\sigma(l^{-1}(t_{1},\dots, t_{s+t}))\cdot\sigma (a)
\]
for $(t_{1},\dots ,t_{s+t})\in l({\mathcal O}_{K}^{\ast\, +})$.
Since $l({\mathcal O}_{K}^{\ast\, +})$ and $\sigma({\mathcal O}_{K})$ are lattices of $L$ and $\R^{s}\times \C^{t}$ respectively, we have an extension $\bar \phi:L\to {\rm Aut} (\R^{s}\times \C^{t}) $ of $\phi$ and $l({\mathcal O}_{K}^{\ast\, +})\ltimes_{\phi} \sigma({\mathcal O}_{K})$ can be seen as a lattice of $L\ltimes_{\bar \phi} (\R^{s}\times \C^{t})$.
By Theorem \ref{MT}, the solvmanifold $L\ltimes_{\bar \phi} (\R^{s}\times \C^{t})/l({\mathcal O}_{K}^{\ast\, +})\ltimes_{\phi} \sigma({\mathcal O}_{K})$ is geometrically formal by Theorem \ref{CCC}.
For a subgroup $U\subset {\mathcal O}_{K}^{\ast\, +}$, we have a Lie group $L^{\prime}\ltimes_{\bar \phi} (\R^{s}\times \C^{t})$ which contains $l(U)\ltimes_{\phi} \sigma({\mathcal O}_{K})$ as a lattice.
The solvmanifold $L^{\prime}\ltimes_{\bar \phi} (\R^{s}\times \C^{t})/l(U)\ltimes_{\phi} \sigma({\mathcal O}_{K})$ is also geometrically formal by Theorem \ref{MT}.
\end{example}
\begin{example}\label{EXC}
Let $G=\R\ltimes_{\phi} U_{3}(\R)$ such that 
\[\phi(t)\left(
\begin{array}{ccc}
1&  x&  z  \\
0&     1&y      \\
0& 0&1 
\end{array}
\right)=\left(
\begin{array}{ccc}
1&  e^{t}x&  z  \\
0&     1&e^{-t}y      \\
0& 0&1 
\end{array}
\right).\]
The left-invariant forms $\bigwedge \g^{\ast}$ on $G$ is generated by $\{ e^{-t}dx,e^{t}dy, dz-xdy, dt\}$.
It is known that $G$ has a lattice $\Gamma$ (see \cite{Saw}).
By simple computations, we have $H^{1}(\g^{\ast})=\langle dt \rangle$, $\dim H^{2}(\g^{\ast})=0$ and $\dim H^{3}(\g^{\ast})=1$.
Since $G$ is completely solvable, we have $H^{\ast}(G/\Gamma, \R)\cong H^{\ast}(\g^{\ast})$ (see \cite{Hatt}) and hence ${\mathcal H}(\g)={\mathcal H}(G/\Gamma)$ where ${\mathcal H}^{\ast}(\g)$ is the set of left-invarinat harmonic forms.
By $d(\bigwedge^{3}\g^{\ast})=0$, for any invariant metric $g$ on $G$, we have:
\[{\mathcal H}^{1}(\g)=\langle dt\rangle , \]
\[{\mathcal H}^{2}(\g)=0,\]
\[{\mathcal H}^{3}(\g)=\langle (\ast_{g} dt)\rangle .\]
Thus any invariant metric on $G/\Gamma$ is formal.
Otherwise we have ${\bf U}_{G}=U_{3}(\C)\times \C$ and hence this solvmanifold is different from examples of geometrically formal solvmanifold given in Theorem  \ref{MT}.

\end{example}

\section{The extension of Theorem \ref{MT}}
Let $G$ be a simply connected solvable Lie group and $g$  an invariant metric which we construct in Section 5.
Denote $C_{g}$  the group of the isometrical automorphisms of $(G, g)$.
Consider $C_{g}\ltimes G$ and the projection $p :C_{g}\ltimes G\to C_{g}$.
\begin{corollary}\label{Inf}
Suppose $G=\R^{n}\ltimes_{\phi} \R^{m}$ with  a semi-simple action $\phi$.
Let $\Gamma\subset C_{g}\ltimes G$ be a torsion-free discrete subgroup such that $G/\Gamma$ is compact.
Suppose $p(\Gamma)$ is finite.

Then the metric $g$ given in the last section is a formal metric on $G/\Gamma$.
\end{corollary}
\begin{proof}
Let $\Delta=\Gamma\cap G$.
Since $\Gamma/\Delta\cong p(\Gamma)$, $\Delta$ is a finite index normal subgroup of $\Gamma$ and  $G/\Delta$ is compact and hence $\Delta\subset G$ is a lattice.
Denote ${\mathcal H}(G/\Gamma)$ and ${\mathcal H}(G/\Delta)$ the sets of the harmonic forms on $G/\Gamma$ and $G/\Delta$ for the metric $g$.
Since we have $A^{\ast}(G/\Gamma)=A^{\ast}(G/\Delta)^{\Gamma/\Delta}$, we have
\[{\mathcal H}(G/\Gamma)= {\mathcal H}(G/\Delta)^{\Gamma/\Delta}.\]
By Theorem \ref{MT}, ${\mathcal H}(G/\Delta)$ is closed under the wedge product, so is ${\mathcal H}(G/\Delta)^{\Gamma/\Delta}$.
Hence the corollary follows.
\end{proof}
\begin{remark}\label{BIB}
Not all cocompact discrete subgroup $\Gamma$ satisfies the assumption of the finiteness of $p(\Gamma)$.
See the following example.
\end{remark}

\begin{example}\label{INO}
Let $G=\R\ltimes_{\phi}\R^{3}$ such that $\phi(t)=\left(
\begin{array}{ccc}
e^{t}&  0&  0 \\
0&     e^{t}&0     \\
0& 0&e^{-2t}
\end{array}
\right)$.
Then $G$ has no lattice (see \cite[Chapter 7]{Hil}).
Consider the metric $g=e^{-2t}dx^2+e^{-2t}dy^2+e^{4t}dz^2+dt^2$.
Then we have $C_{g}=O(2)\times O(1)$ acting as rotations and reflections on the $(x,y)$-coordinates and reflection on the $z$-coordinate.
$C_{g}\ltimes G$ admits a torsion-free cocompact discrete subgroup $\Gamma$.
Since $G\cap \Gamma$ is not a lattice of $G$, $p(\Gamma)$ is not finite.
In \cite[Chapter 8]{Hil} it is proved that $\Gamma\cong \Z\ltimes_{\phi}\Z^{3}$ and for $t\not=0$ $\phi(t)\in SL_{3}(\Z)$ has a pair of complex conjugate eigenvalues (see \cite[Chapter 7]{Hil}).
Hence $\Gamma$ can be a lattice of  a Lie group $H=\R\ltimes _{\phi}\R^{3} $ with $\phi(t)=\left(\begin{array}{ccc} e^{t}\cos c t & -e^{t}\sin c t &0\\ e^{t}\sin c t &e^{t}\cos c t&0\\
0&0& e^{-2t}
\end{array}\right)
 $, and $G/\Gamma=H/\Gamma$ is geometrically formal by Theorem \ref{MT}.
\end{example}
\section{Thurston's Geometries and infrasolvmanifold}
We say that a compact oriented manifold $M$ admits a geometry $(X, g)$  if $M=X/\Gamma$ where $X$ is a simply connected manifold with a complete Riemaniann metric $g$ and $\Gamma$ is a cocompact discrete subgroup of the group $Isom_{g}(X)$ of isometries.
If $(X,g)$ is a solvable Lie group with an invariant metric $g$, we call it a solvable Lie type geometry.
We consider the following $3$-dimensional solvable Lie type geometries.\\
(3-A) $X=E^{3}= \R^{3}$, $g_{E^{3}}=dx^{2}+dy^{2}+dz^{2}$.\\
(3-B) $X=Nil^{3}=U_{3}(\R)=\left\{\left(
\begin{array}{ccc}
1&  x&  z  \\
0&     1&y      \\
0& 0&1 
\end{array}
\right): x,y,z\in \R\right\}$, $g_{Nil^3}=dx^{2}+dy^{2}+(dz-xdy)^{2}$.\\
(3-C) $X=Sol^{3}=\R\ltimes_{\phi}\R^{2}$ with $\phi(z)=\left(
\begin{array}{cc}
e^t&  0 \\
0&   e^{-t}  
\end{array}
\right)$, $g_{Sol^{3}}=e^{2z}dx^{2}+e^{-2z}dy^{2}+dz^{2}$. 

By the theory of geometry and topology of $3$-dimensional manifolds we have the following theorem (see \cite{PS}).
\begin{theorem}\label{333}
A compact aspherical $3$-dimensional manifold  with the  virtually solvable fundamental group admits  one of the geometries (3-A$\sim$C). 

\end{theorem}

We also consider the following $4$-dimensional solvable Lie type geometries (listed in \cite{W}).\\
(4-A) $X=E^{4}=\R^{4}$, $g_{E^{4}}=dx^{2}+dy^{2}+dz^{2}+dt^{2}$.\\
(4-B)  $X=Nil^3\times E=U_{3}(\R)\times \R$, $g_{Nil^3\times E}=dx^{2}+dy^{2}+(dz-xdy)^{2}+dt^{2}$.\\
(4-C)   $X=Nil^{4}=\left\{\left(
\begin{array}{cccc}
1&  t&  \frac{1}{2}t^{2}&z  \\
0&     1&t & y    \\
0& 0&1 &x \\
0&0&0&1
\end{array}
\right): x,y,z,t\in \R\right\}$,\\ $g_{Nil^4}=dx^2+(dy-tdz)^2+(dz-tdy+\frac{1}{2}t^2dx)^2+dt^2$\\
(4-D) $X=Sol^{3}\times E$, $g_{Sol^{3}\times E}=e^{2z}dx^{2}+e^{-2z}dy^{2}+dz^{2}+dt^{2}$.\\
(4-E) $X=Sol^{4}_{m,n}=\R\ltimes_{\phi}\R^{3}$ such that $\phi(t)=\left(
\begin{array}{ccc}
e^{at}&  0&  0 \\
0&     e^{bt}&0     \\
0& 0&e^{ct}
\end{array}
\right)$, where $e^a, e^{b}, e^{c}$ are distinct roots of $X^{3}-mX^2+nX-1$ for real numbers $a<b<c$ and integers $m<n$,
$g_{Sol^{4}_{m,n}}=e^{-2at}dx^2+e^{-2bt}dy^2+e^{-2ct}dz^2+dt^2$.\\
(4-F) $X=Sol^{4}_{0}=\R\ltimes_{\phi}\R^{3}$ such that $\phi(t)=\left(
\begin{array}{ccc}
e^{t}&  0&  0 \\
0&     e^{t}&0     \\
0& 0&e^{-2t}
\end{array}
\right)$, \\
$g_{Sol^{4}_{0}}=e^{-2t}dx^2+e^{-2t}dy^2+e^{4t}dz^2+dt^2$.\\
(4-G) $X=Sol_{1}^{4}=\R\ltimes_{\phi} U_{3}(\R)$ such that 
$\phi(t)\left(
\begin{array}{ccc}
1&  x&  z  \\
0&     1&y      \\
0& 0&1 
\end{array}
\right)=\left(
\begin{array}{ccc}
1&  e^{t}x&  z  \\
0&     1&e^{-t}y      \\
0& 0&1 
\end{array}
\right)$,\\
$g_{Sol^{4}_{1}}=e^{-2t}dx^2+e^{2t}dy^2+ (dz-xdy)^2+dt^2$.

Let $G$ be a simply connected solvable Lie group and $g$ an invariant metric on $G$.
We consider the affine transformation group ${\rm Aut}(G)\ltimes G$ and the projection $p:{\rm Aut}(G)\ltimes G\to{\rm Aut}(G)$.
Let $\Gamma\subset {\rm Aut}(G)\ltimes G$ be a torsion-free discrete subgroup such that $p(\Gamma)$ is contained in a compact subgroup of ${\rm Aut}(G)$ and the quotient $G/\Gamma$ is compact.
We call $G/\Gamma$ an infra-solvmanifold.
If $G$ is nilpotent, $G/\Gamma$ is called an infra-nilmanifold.
Since $\Gamma \subset Isom_{g}(G)$ does not satisfies $\Gamma\subset {\rm Aut}(G)\ltimes G$, a compact manifold with a solvable Lie type geometry is not an infra-solvmanifold in general.
Suppose $Isom_{g}(G)\subset  {\rm Aut}(G)\ltimes G$.
Then for an isometry transformation $(\phi, x)\in  {\rm Aut}(G)\ltimes G$, $\phi$ is an also isometry transformation.
By this, for the group $C_{g}$ of the isometrical automorphisms of $G$, we have $Isom_{g}(G)=C_{g}\ltimes G$.
Thus in the assumption $Isom_{g}(G)\subset  {\rm Aut}(G)\ltimes G$, a compact manifold with a solvable Lie type geometry is  an infra-solvmanifold.
It is known that for the Euclidian geometry $(E^{n}, g_{E^{n}}=dx^2_{1}+\dots+ dx_{n}^{2})$ we have $Isom_{g_{E^{n}}}=O(n)\ltimes \R^{n}$ and the geometries (3-A$\sim$C) satisfies $Isom_{g}(G)\subset  {\rm Aut}(G)\ltimes G$(see \cite{PS}).
In \cite{Hil}, Hillman studied the structures of $Isom_{g}(G)$ of the geometries (4-A$\sim$H) and proved $Isom_{g}(G)\subset  {\rm Aut}(G)\ltimes G$.
In \cite{Hil2}, Hillman proved the following theorem.
\begin{theorem}{\rm (\cite[Theorem 8]{Hil2})} \label{HILL}
A $4$-dimensional infra-solvmanifold is diffeomorphic to a manifold which admits one of the geometries (4-A$\sim$G).
\end{theorem}

\begin{remark}\label{InfB}
By Baues's result in \cite{B}, any compact  aspherical manifold with the virtually solvable fundamental group is homotopy equivalent to an infra-solvmanifold $G/\Gamma$.
But for dimension$\ge 4$,  there may exist a compact aspherical manifold with virtually solvable fundamental group which is not diffeomorphic to an infra-solvmanifold.
\end{remark}

\section{Geometrical formality of $3$-manifolds}

\begin{theorem}\label{33}
Let $M$ be a compact oriented  aspherical  $3$-manifold with  the virtually solvable fundamental group.
If $M$ is a torus or not a nilmanifold, then $M$ is geometrically formal.
\end{theorem}
\begin{proof}
By Theorem \ref{333}, it is sufficient to consider the geometries (3-A$\sim$C). 
In the case (3-A),  by  Corollary \ref{Inf} and the first  Bieberbach theorem $g_{E^{3}}$ is a formal metric on $G/\Gamma$.

In the case (3-C), it is known that $C_{g}$ is isomorphic to the finite dihedral group  $D(8)$ (see \cite{PS}) and  hence by  Corollary \ref{Inf} $g_{Sol^{3}}$ is a formal metric on $G/\Gamma$.

Suppose $(G,g)$ is in the case (3-B).
Then  $C_{g}$ has two components and the identity component of $C_{g}$ is isomorphic to a circle $S^{1}$.
Let $\Delta=\Gamma\cap G$.
By Generalized Bieberbach's theorem (see \cite{Au}), $\Delta$ is a finite  index normal subgroup of  $\Gamma$.
Consider the projection $p: C_{g}\ltimes G \to C_{g}$.
If $p(\Gamma)$ is trivial, then $\Gamma\subset G$ is a lattice and $G/\Gamma$ is a non-toral nilmanifold and hence not formal (see \cite{H}).
Suppose $p(\Gamma) $ is non-trivial.
By Nomizu's theorem (\cite{Nom}) we have 
\[H^{\ast}(G/\Delta, \R)\cong H^{\ast}(\g)\]
where $\g$ is the Lie algebra of $G$.
By this we have 
\[H^{\ast}(G/\Gamma, \R)\cong H^{\ast}(G/\Delta, \R)^{\Gamma/\Delta}\cong H^{\ast}(\g)^{\Gamma/\Delta}.
\]
In \cite[Lemma 13.1]{BG}, it is shown that a non-trivial semisimple automorphism of a nilpotent Lie algebra $\g$ acts non-trivially on $H^{1}(\g)$.
Since $\Gamma/\Delta \cong p(\Gamma)$ is a nontrivial finite group,
\[H^{1}(\g)^{\Gamma/\Delta}\not= H^{1}(\g).
\]
Since $\dim H^{1}(\g)=2$, $\dim  H^{1}(\g)^{\Gamma/\Delta}=0$ or $1$.
If $  \dim  H^{1}(\g)^{\Gamma/\Delta}=0$, then $G/\Gamma$ is a  rational homology sphere and any metric on $G/\Gamma$ is formal.
Suppose  $\dim  H^{1}(\g)^{\Gamma/\Delta}=1$.
Then $b_{i}=1$ for any $1\le i\le 3$.
For any $1\le i\le 3$, we have
\[H^{\ast}(G/\Gamma,\R)\cong   H^{\ast}(\g)^{\Gamma/\Delta}=\bigoplus_{i=1}^{3}\langle [\alpha_{i}]\rangle\]
for non-zero cohomology classes $[\alpha_{i}]\in  H^{i}(\g)$.
We can choose invariant harmonic forms $\alpha_{i}$, $i=1,2,3$  for the invariant metric $g$.
Then we have ${\mathcal H}(G/\Gamma)=\bigoplus_{i=1}^{3}\langle \alpha_{i}\rangle$,
Since all elements of $\bigwedge ^{3}\g^{\ast}$ are harmonic, $\alpha_{1}\wedge\alpha_{2}$ is harmonic.
For  $i<j$ with $(i,j)\not=(1,2)$, we have $\alpha_{i}\wedge\alpha_{j}=0$.
Thus $g$ is a formal metric on $G/\Gamma$.
This completes the proof of the theorem.
\end{proof}
\begin{remark}
There exists a closed $3$-dimensional infra-nilmanifold which is not a nilmanifold.
By this theorem such a manifold is geometrically formal.
\begin{example}
Consider $\Gamma=\Z\ltimes_{\phi}\Z^{2}$ such that $\phi(t)=\left(
\begin{array}{cc}
(-1)^t&  (-1)^t t \\
0&   (-1) ^t 
\end{array}
\right)$.
Then we can embed $\Gamma$ in $Isom_{g_{Nil^{3}}}( Nil^{3})$ (see \cite{K}).
By the direct computation of the lower central series, $\Gamma$ is non-nilpotent and hence $Nil^{3}/\Gamma$ is not a nilmanifold.

\end{example}
\end{remark}

\section{Aspherical manifolds with the virtually solvable fundamental groups}
\begin{theorem}\label{44}
Let $M$ be an oriented $4$-dimensional infra-solvmanifold.
If $M$ is a torus or not a nilmanifold, then $M$ is geometrically formal.
\end{theorem}
\begin{proof}
In the case (4-A), by  Corollary \ref{Inf} and the first  Bieberbach theorem $g_{E^{4}}$ is a formal metric on $G/\Gamma$.

In the case (4-D) (resp (4-E)), $C_{g}$ is isomorphic to the finite group $D(8)\times (\Z/2\Z)$ (resp. $(\Z/2\Z)^{3}$) (see \cite[Chapter 7]{Hil}) and hence $g_{Sol^{3}\times E}$ (resp. $g_{Sol^{4}_{m,n}}$) is a formal metric on $G/\Gamma$ by Corollary \ref{Inf}.

As we  showed in Example \ref{INO}, in the case (4-F) $G/\Gamma$ is geometrically formal.

In the case (4-B), the group of all the orientation preserving isomorphisms is $Isom_{g_{Nil^{3}} }Nil^{3}\times \R$ (see \cite{Ue}, \cite{W}, or \cite{W2}).
Thus as the proof of Theorem  \ref{33} for the case (3-B), if $G/\Gamma$ is a nilmanifold then $G/\Gamma$ is not formal, and if $G/\Gamma$ is an infra-nilmanifold but not a nilmanifold then $g_{Nil^{3}\times \R}$ is formal.

In the case  (4-C), the group of all the orientation preserving isomorphisms is $Nil^{4}$ itself (see \cite{Ue}, \cite{W}, or \cite{W2}).
Thus  oriented $Nil^{4}$ manifolds are only nilmanifolds and so all  oriented $Nil^{4}$ manifolds are not formal.

In the case (4-G) we have $Isom_{g_{Sol^{4}_{1}}}(Sol^{4}_{1})\cong D(4)\ltimes Sol^{4}_{1}$.
For any cocompact discrete subgroup $\Gamma\subset Isom_{g_{Sol^{4}_{1}}}(Sol^{4}_{1})$, since for the projection $p:{\rm Aut} G\ltimes G \to {\rm Aut} (G)$, $p(\Gamma)\subset D(4)$ is finite, we have a subgroup $\Delta\subset \Gamma$ which is a lattice of $ Sol^{4}_{1}$ and we have ${\mathcal H} (Sol^{4}_{1}/\Gamma)= {\mathcal H} (Sol^{4}_{1}/\Delta)^{\Gamma/\Delta}$.
In Example \ref{EXC}, we showed that the metric $g_{Sol^{4}_{1}}$ on the solvmanifold $Sol^{4}_{1}/\Delta$ is formal.
Thus the metric $g_{Sol^{4}_{1}}$ on every $Sol^{4}_{1}$ manifold  is formal.
Hence the theorem follows.

\end{proof}

Finally we prove:
\begin{theorem}
Let $M$ be a compact oriented aspherical manifold of dimension less than or equal to $4$ with the virtually solvable fundamental group.
Then  $M$ is geometrically formal if and only if $M$ is diffeomorphic to a torus or an infra-solvmanifold which is not a nilmanifold.
\end{theorem}
\begin{proof}
It is sufficient to show that $M$ is an infra-solvmanifold if $M$ is geometrically formal.
If $\dim M\le 2$, it is obvious.
If  $\dim M=3$, it follows from Theorem \ref{333}.
We consider $\dim M=4$.
As Remark \ref{InfB}, $M$ is homotopy equivalent to an infra-solvmanifold $G/\Gamma$.
It is known that the Euler characteristic $\chi(G/\Gamma)$ of an infra-solvamanifold is $0$ (see  \cite[Cahpter 8]{Hil}).
Since $G/\Gamma$ is an oriented $4$-manifold, $\chi(G/\Gamma)=0$ implies $b_{1}(G/\Gamma)\not=0$.
Thus we have $b_{1}(M)\not=0$.
If $M$ is geometrically formal, we have a submersion $M\to T^{b_{1}(M)}$ (see \cite[Theorem 7]{Kot}) and hence $M$ is a fiber bundle over a torus $T^{b_{1}(M)}$.
Now  we suppose that $M$ is a compact oriented aspherical manifold of dimension $4$ with the virtually solvable fundamental group.
 By the exact sequence of homotopy groups associated by the fiber bundle, the fiber of $M\to T^{b_{1}(M)}$ is a compact aspherical manifold of dimension less than or equal $3$ with the virtually solvable fundamental group, and hence it is an infra-solvmanifold.
Thus $M$ is a fiber bundle whose fiber is an infra-solvmanifold and base space is a torus.
By  \cite[Theorem 7]{Hil2}, $M$ is diffeomorphic to an infra-solvmanifold with the fundamental group $\pi_{1}(M)$.
\end{proof}

\ \\
{\bf  Acknowledgements.} 
The author would like to express many thanks to the referee for his careful reading of the earlier version of manuscript with several important remarks, which lead to many improvements in the revised version.

\end{document}